\documentclass[12pt]{article}

\usepackage{amsthm,amsfonts, amsbsy, amssymb,amsmath,graphicx}
\usepackage{graphics}
\usepackage[english]{babel}
\usepackage{graphicx}
\usepackage{lineno}
\usepackage{enumerate}
\usepackage{tikz}
\usepackage{tkz-graph}
\usepackage{tkz-berge}
\usepackage{hyperref}
\usepackage{color}
\usepackage{tikz}

\hypersetup{colorlinks=true}
\hypersetup{colorlinks=true, linkcolor=blue, citecolor=blue,urlcolor=blue}

\newtheorem{theorem}{Theorem}
\newtheorem{lemma}[theorem]{Lemma}
\newtheorem{proposition}[theorem]{Proposition}
\newtheorem{corollary}[theorem]{Corollary}

\newtheorem{remark}[theorem]{Remark}

\title{On computational and combinatorial properties of the total co-independent domination number of graphs}

\author{Abel Cabrera Mart\'inez$^{(1)}$, Frank A. Hern\'andez Mira $^{(1)}$\\ Jos\'e M. Sigarreta Almira$^{(1)}$ and Ismael G. Yero$^{(2)}$\\
\\
$^{(1)}$ {\small Facultad de  Matem\'{a}ticas, Universidad Aut\'onoma de Guerrero}\\
{\small Carlos E. Adame 5, Col. La Garita, Acapulco, Guerrero, Mexico}\\
{\small\it abecamar\@@gmail.com}, {\small\it fmira8906\@@gmail.com}, {\small\it josemariasigarretaalmira\@@hotmail.com}\\
$^{(2)}${\small Departamento de Matem\'aticas, Escuela Polit\'ecnica Superior de Algeciras}\\
{\small Universidad de C\'adiz,} {\small
Av. Ram\'on Puyol s/n, 11202 Algeciras, Spain.} \\ {\small\it
ismael.gonzalez\@@uca.es}
}

\date{}

\begin{document}

\maketitle

\begin{abstract}
A subset $D$ of vertices of a graph $G$ is a total dominating set if every vertex of $G$ is adjacent to at least one vertex of $D$. The total dominating set $D$ is called a total co-independent dominating set if the subgraph induced by $V-D$ is edgeless and has at least one vertex. The minimum cardinality of any total co-independent dominating set is the total co-independent domination number of $G$ and is denoted by $\gamma_{t,coi}(G)$. In this work we study some complexity and combinatorial properties of $\gamma_{t,coi}(G)$. Specifically, we prove that deciding whether $\gamma_{t,coi}(G)\le k$ for a given integer $k$ is an NP-complete problem and give several bounds on $\gamma_{t,coi}(G)$. Also, since any total co-independent dominating set is also a total dominating set, we characterize all the trees having equal total co-independent domination number and total domination number.
\end{abstract}



\section{Introduction} \label{Intro}

Problems concerning domination in graphs are one of the most popular and highly investigated ones in the area of graph theory, and a rich literature in the topic is nowadays known in the research community. Such problems run from the theoretical point of view till several practical applications in real life situations. The most interesting cases of such applications are probably regarding problems in computer science. One of the most interesting features of domination in graphs involves the existence of a very high number of variants of domination parameters. A very common combination of domination is made with vertex independence of graphs, and the way of combining both concepts groups a considerable number of possibilities. In this work, we center our attention into precisely study one of these combinations between domination and independence in graphs, namely the total co-independence domination parameter. We focus the investigation on some computational complexity aspects of this parameter as well as on combinatorial properties of it.

Given a graph $G$ with vertex set $V(G)$ and edge set $E(G)$, a set $D\subset V(G)$ is a \emph{total dominating set} of $G$ if every vertex in $V(G)$ is adjacent to at least one vertex in $D$.  The \emph{total domination number} of $G$ is the minimum cardinality of any total dominating set in $G$ and is denoted by $\gamma_t(G)$. A $\gamma_t(G)$-set is a total dominating set of cardinality $\gamma_t(G)$. For more information on total domination we suggest the recent and fairly complete survey \cite{Henning2009} and the book \cite{book-total-dom}. A set $S$ of vertices is \emph{independent} if $S$ induced an edgeless graph. An independent set of maximum cardinality is a \emph{maximum independent set} of $G$. The \emph{independence number} of $G$ is the cardinality of a maximum independent set of $G$ and is denoted by $\beta(G)$. An independent set of cardinality $\beta(G)$ is called a $\beta(G)$-\emph{set}. Relationships between (total) domination and independence in graphs have attracted the attention of several researchers in the last years. Several interesting connection among these parameters include independent dominating sets \cite{berge,ore}, partitions into a dominating set and an independent set \cite{Lowenstein}, (total) dominating sets which intersect every maximal independent set \cite{yero,yero2,hamid}, and some other ones more, which we prefer to not mention here, since it is not the goal of this work.

A total dominating set $D$ of a graph $G$ is called a \emph{total co-independent dominating set} (or TC-ID set for short) if the set of vertices of the subgraph induced by $V-D$ is independent and not empty\footnote{Notice that the condition of $V-D$ to be not empty is not exactly necessary. However, if such condition is not required, then we readily seen that the only graphs containing a TC-ID set of minimum cardinality with empty complement are the union of paths $P_2$.}. The minimum cardinality of any TC-ID set is the \emph{total co-independent domination number} of $G$ and is denoted by $\gamma_{t,coi}(G)$. A TC-ID set of cardinality $\gamma_{t,coi}(G)$ is a $\gamma_{t,coi}(G)$-\emph{set}. These concepts were previously introduced and barely studied in \cite{Soner2012}. Moreover, in \cite{Marcin-K}, the same parameter was introduced under the name of total outer-independent domination number. Since this article (\cite{Marcin-K}) is not published in any journal and the other one (\cite{Soner2012}) is already published, we precisely follow the terminology and notation of the latter. Since total domination is not defined for graphs having isolated vertices, all the graphs considered herein have not isolated vertices. Moreover, in order to satisfy the total domination property and that the complement of a TC-ID set is not empty, it is required that $2\le \gamma_{t,coi}(G)\le n-1$, if $n$ is the order of $G$. Such trivial bounds were already noted in the seminal work \cite{Soner2012}.

Throughout this work we consider $G=(V,E)$ as a simple graph of order $n$ and size $m$. That is, graphs that are finite, undirected, and without loops or multiple edges. Given a vertex $v$ of $G$, $N_G(v)$ represents the \emph{open neighborhood} of $v$, \emph{i.e.}, the set of all neighbors of $v$ in $G$ and the \emph{degree} of $v$ is $\delta(v) = |N_G(v)|$. The \emph{minimum} and \emph{maximum degrees} of $G$ are denoted by $\delta(G)$ and $\Delta(G)$, respectively (or $\delta$ and $\Delta$, respectively, for short). If $X$ and $Y$ are two subsets of $V(G)$, then we denote the set of all edges of $G$ joining a vertex of $X$ with a vertex of $Y$ by $E(X,Y)$. For a set $S\subset V(G)$, the \emph{complement} of $S$ is $\overline{S}=V(G)\setminus S$. In this work, we represent an edgeless graph $G$ of order $n$ as $N_n$. For any other graph theory terminology and notation we follow the book  \cite{book-total-dom}.

Let $T$ be a tree (a connected graph without cycles). A \emph{leaf} or a \emph{pendant vertex} of $T$ is a vertex of degree one (it is similarly defined for non tree graphs). A \emph{support vertex} of $T$ is a vertex adjacent to a leaf and a \emph{semi-support vertex} is a vertex adjacent to a support vertex that is not a leaf. By an \emph{isolated support vertex} of $T$ we mean an isolated vertex of the subgraph induced by the support vertices of $T$. The set of leaves of $T$ is denoted by $L(T)$, the set of support vertices by $S(T)$, and the set of semi-support vertices by $SS(T)$. Moreover, $S^\ast(T)$ is the set of isolated support vertices of $T$.

We first notice that if $H_1$, $H_2$, \ldots , $H_r$ with $r\ge 2$, are the connected components of a graph $H$, then any TC-ID set of minimum cardinality in $H$ is formed by a minimum total dominating set in the subgraphs $H_j$ where $|V(H_j)|=2$ and a minimum TC-ID set in the remaining subgraphs $H_i$ with $|V(H_i)|\ge 3$. That is stated in the following result.

\begin{remark}
Let $H_1$, $H_2$, \ldots , $H_r$ with $r\ge 2$, be the connected components of a graph $H$ different from the union of $r$ copies of the path $P_2$. Then $$\gamma_{t,coi}(H)=\sum_{\substack{i\in \{1, \ldots, r\}\\ |V(H_i)|=2}}\gamma_t(H_i) + \sum_{\substack{j\in \{1, \ldots, r\}\\ |V(H_j)|\geq 3}}\gamma_{t,coi}(H_j).$$
\end{remark}

\begin{proof}
Let $D_j$ be a $\gamma_{t,coi}(H_j)$-set for $j\in\{1, \ldots, r\}$ such that $|V(H_j)|\geq 3$. It is easy see that $(\bigcup_{\substack{i\in \{1, \ldots, r\}\\ |V(H_i)|=2}}V(H_i))\cup(\bigcup_{\substack{j\in \{1, \ldots, r\}\\ |V(H_j)|\geq 3}}D_j)$ is a TC-ID set of $H$ and we have

$$\gamma_{t,coi}(H)\leq \sum_{\substack{i\in \{1, \ldots, r\}\\ |V(H_i)|=2}}\gamma_t(H_i) + \sum_{\substack{j\in \{1, \ldots, r\}\\ |V(H_j)|\geq 3}}\gamma_{t,coi}(H_j).$$

On the other hand, let $A$ be a $\gamma_{t,coi}(H)$.  Firstly, we observe that for every $i\in \{1, \ldots, r\}$ such that $|V(H_i)|=2$, it is satisfied that $A\cap V(H_i)=V(H_i)$. Moreover, let $A_j=A\cap V(H_j)$ for every $j\in \{1, \ldots, r\}$ such that $|V(H_j)|\geq 3$. We notice that every $A_j$ must be a TC-ID set of $H_j$. In this sense,
\begin{align*}
\gamma_{t,coi}(H)=|A|&=\sum_{\substack{i\in \{1, \ldots, r\}\\ |V(H_i)|=2}}|V(H_i)|+\sum_{\substack{j\in \{1, \ldots, r\}\\ |V(H_j)|\geq 3}}|A_j|\\
&\ge \sum_{\substack{i\in \{1, \ldots, r\}\\ |V(H_i)|=2}}\gamma_t(H_i)+\sum_{\substack{j\in \{1, \ldots, r\}\\ |V(H_j)|\geq 3}}\gamma_{t,coi}(H_j),
\end{align*}
which completes the proof.
\end{proof}

In concordance with the result above, from now on, we only consider the study of the TC-ID sets of connected graphs and omit to refer to that fact throughout all our exposition.

\section{Complexity of the decision problem}

We begin our exposition by considering the problem of deciding whether the total co-independent domination number of a graph is less than a given integer. That is stated in the following decision problem.

$$\begin{tabular}{|l|}
  \hline
  \mbox{TOTAL CO-INDEPENDENT DOMINATION PROBLEM}\\
  \mbox{INSTANCE: A non trivial graph $G$ and a positive integer $r$}\\
  \mbox{PROBLEM: Deciding whether $\gamma_{t,coi}(G)$ is less than $r$}\\
  \hline
\end{tabular}$$

In order to deal with the complexity of the TOTAL CO-INDEPENDENT DOMINATION PROBLEM (TC-ID PROBLEM), we make a reduction from a very well known decision problem concerning the independence number of graphs.

$$\begin{tabular}{|l|}
  \hline
  \mbox{MAXIMAL INDEPENDENT SET PROBLEM}\\
  \mbox{INSTANCE: A non trivial graph $G$ and a positive integer $r$}\\
  \mbox{PROBLEM: Deciding whether the independence number of $G$ is larger than $r$}\\
  \hline
\end{tabular}$$

The problem above is one of the classical NP-complete problems appearing in the book \cite{garey}. Moreover, it remains NP-complete even when restricted to planar graphs.

\begin{theorem}{\em \cite{garey}}
MAXIMAL INDEPENDENT SET PROBLEM is NP-complete even when restricted to planar graphs of maximum degree at most 3.
\end{theorem}

Now on, in order to present our complexity results we need to introduce a family of graphs which is next defined. Let $T_6$ be a tree with six vertices having two adjacent vertices $u,v$ of degree three and the other four $u_1,u_2,v_1,v_2$ vertices are leaves. Clearly, each vertex of degree three has two adjacent leaves, say $u_1,u_2\in N(u)$ and $v_1,v_2\in N(v)$ (see Figure \ref{G-T-6} (I)). Given a graph $G$ of order $n$ and $n$ trees $T_6^{(1)},\dots,T_6^{(n)}$ isomorphic to the tree $T_6$, the graph $G_T$ is constructed by adding edges between the $i^{th}$-vertex of $G$ and the vertex $u$ of the $i^{th}$-tree $T_6^{(i)}$. See Figure \ref{G-T-6} (II) for an example.

\begin{figure}[h]
\centering
\begin{tikzpicture}[scale=.8, transform shape]

\node [draw, shape=circle] (t1) at  (-4,1) {};
\node [draw, shape=circle] (t2) at  (-4.5,2) {};
\node [draw, shape=circle] (t3) at  (-3,2) {};
\node [draw, shape=circle] (t4) at  (-3.5,3) {};
\node [draw, shape=circle] (t5) at  (-4,4) {};
\node [draw, shape=circle] (t6) at  (-2.5,4) {};

\draw(t2)--(t1)--(t3);
\draw(t5)--(t4)--(t6);
\draw(t4)--(t1);

\node [below] at (-4,0.7) {$u$};
\node [right] at (-4.2,3) {$v$};
\node [right] at (-2.8,2) {$u_2$};
\node [right] at (-2.3,4) {$v_2$};
\node [left] at (-4.7,2) {$u_1$};
\node [left] at (-4.2,4) {$v_1$};

\node [left] at (-2,0.5) {\large (I)};
\node [left] at (12,0.5) {\large (II)};

\node [draw, shape=circle] (a1) at  (0,0) {};
\node [draw, shape=circle] (a2) at  (3,0) {};
\node [draw, shape=circle] (a3) at  (6,0) {};
\node [draw, shape=circle] (a4) at  (9,0) {};
\node [draw, shape=circle] (a5) at  (1,1) {};
\node [draw, shape=circle] (a6) at  (4,1) {};
\node [draw, shape=circle] (a7) at  (7,1) {};
\node [draw, shape=circle] (a8) at  (10,1) {};
\node [draw, shape=circle] (a9) at  (2,2) {};
\node [draw, shape=circle] (a10) at  (5,2) {};
\node [draw, shape=circle] (a11) at  (8,2) {};
\node [draw, shape=circle] (a12) at  (11,2) {};
\node [draw, shape=circle] (b9) at  (0.5,2) {};
\node [draw, shape=circle] (b10) at  (3.5,2) {};
\node [draw, shape=circle] (b11) at  (6.5,2) {};
\node [draw, shape=circle] (b12) at  (9.5,2) {};
\node [draw, shape=circle] (b13) at  (1.5,3) {};
\node [draw, shape=circle] (b14) at  (4.5,3) {};
\node [draw, shape=circle] (b15) at  (7.5,3) {};
\node [draw, shape=circle] (b16) at  (10.5,3) {};
\node [draw, shape=circle] (a17) at  (1,4) {};
\node [draw, shape=circle] (a18) at  (4,4) {};
\node [draw, shape=circle] (a19) at  (7,4) {};
\node [draw, shape=circle] (a20) at  (10,4) {};
\node [draw, shape=circle] (b17) at  (2.5,4) {};
\node [draw, shape=circle] (b18) at  (5.5,4) {};
\node [draw, shape=circle] (b19) at  (8.5,4) {};
\node [draw, shape=circle] (b20) at  (11.5,4) {};

\draw(a9)--(a5)--(a1)--(a2)--(a3)--(a4)--(a8)--(a12);
\draw(a2)--(a6)--(a10);
\draw(a3)--(a7)--(a11);
\draw(b13)--(a5)--(b9);
\draw(b14)--(a6)--(b10);
\draw(b15)--(a7)--(b11);
\draw(b16)--(a8)--(b12);
\draw(a17)--(b13)--(b17);
\draw(a18)--(b14)--(b18);
\draw(a19)--(b15)--(b19);
\draw(a20)--(b16)--(b20);

\draw (a1) .. controls (0.3,-0.5) and (5.7,-0.5) .. (a3);
\draw (a2) .. controls (3.3,-0.5) and (8.7,-0.5) .. (a4);
\end{tikzpicture}
\caption{The graph $T_6$ (I) and a graph $G_T$ (II) where $G$ is a complete graph minus one edge.}
\label{G-T-6}
\end{figure}
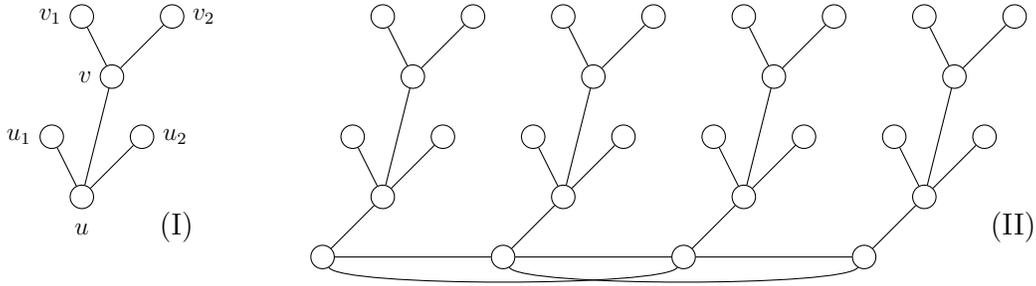

We are now able to prove the NP-completeness of the TC-ID PROBLEM.

\begin{theorem}
TOTAL CO-INDEPENDENT DOMINATION PROBLEM is NP-complete even when restricted to planar graphs of maximum degree at most 3.
\end{theorem}

\begin{proof}
The problem is clearly in NP since verifying that a given set is indeed a TC-ID set can be done in polynomial time. Let us now make a reduction from the MAXIMAL INDEPENDENT SET PROBLEM. Let $G$ be a not edgeless graph of order $n$ and construct the graph $G_T$ as described above. Let us denote by $u^{(i)}, v^{(i)}$ the vertices of degree three in the $i^{th}$ copy $T_6^{(i)}$ of the tree $T_6$ used to generate $G_T$. We shall prove that $\gamma_{t,coi}(G_T)=3n-\beta(G)$.

Let $A$ be a $\beta(G)$-set and let $D$ be the set of vertices of $G_T$ obtained from the complement of $A$ in $G$ together with the vertices $u,v$ belonging to all the copies of the tree $T_6$ used to generate $G_T$, that is $D=(V(G)\setminus A)\cup \left\{\bigcup_{i=1}^n\{u^{(i)},v^{(i)}\}\right\}$. This set is clearly a total dominating set and its complement is an independent set. Thus, $D$ is a TC-ID set in $G_T$ and, as a consequence, $$\gamma_{t,coi}(G_T)\le n-|A|+\left|\bigcup_{i=1}^n\{u^{(i)},v^{(i)}\}\right|=3n-\beta(G).$$
On the other hand, let $D'$ be a $\gamma_{t,coi}(G_T)$-set. In order to totally dominate the leaves of every copy of $T_6$ in $G_T$, it must happen that $|D'\cap V(T_6^{(i)})|\ge 2$ for every $i\in\{1,\dots,n\}$. Moreover, $V(G)\cap D'\ne \emptyset$, since otherwise the complement of $D'$ would not be independent. Moreover the complement of $V(G)\cap D'$ in $G$ is an independent set in $G$. Thus, $\beta(G)\ge n-|V(G)\cap D'|$ and we obtain the following.
$$\gamma_{t,coi}(G_T)= |D'|= |D'\cap V(G)| + \left|\bigcup_{i=1}^n \left(D'\cap V(T_6^{(i)})\right)\right|\ge n-\beta(G)+2n=3n-\beta(G).$$
As a consequence, it follows that $\gamma_{t,coi}(G_T)=3n-\beta(G)$.

Now, for $j=3n-k$, it is readily seen that $\gamma_{t,coi}(G_T)\leq j$ if and only if $\beta(G)\geq k$, which complete the reduction. We also observe that, if $G$ is a planar graph, then $G_T$ is also planar. Therefore, since the MAXIMAL INDEPENDENT SET PROBLEM is NP-complete even when restricted to planar graphs of maximum degree at most 3, we also deduce that the TC-ID PROBLEM is NP-complete even when restricted to planar graphs of maximum degree at most 3 and the proof is completed.
\end{proof}

As a consequence of the result above, we deduce the following consequence.

\begin{corollary}
The problem of computing the total co-independent domination number of graphs is NP-hard even when restricted to planar graphs of maximum degree at most 3.
\end{corollary}

\section{Bounding the total co-independent domination number}

In order to present the first bounds for $\gamma_{t,coi}(G)$ of any graph $G$, we need the next concepts. A set $S$ of vertices of $G$ is a \emph{vertex cover} of $G$ if every edge of $G$ is incident with at least one vertex of $S$. The \emph{vertex cover number} of $G$, denoted by $\alpha(G)$, is the smallest cardinality of a vertex cover of $G$. We refer to an $\alpha(G)$-set in $G$ as a vertex cover of cardinality $\alpha(G)$. The following well-known result, due to Gallai \cite{Gallai1959}, states the relationship between the independence number and the vertex cover number of a graph.

\begin{theorem}{\em\cite{Gallai1959}}{\rm (Gallai, 1959)}\label{th gallai}
For any graph $G$ of order $n$, $\alpha(G)+\beta(G) = n.$
\end{theorem}

On the other hand, it was shown in \cite{Soner2012} the following relationship between $\gamma_{t,coi}(G)$ and $\beta(G)$.

\begin{theorem}{\em\cite{Soner2012}}\label{th-indep-tcid}
For any graph $G$ of order $n$, $\gamma_{t,coi}(G)\ge n-\beta(G).$
\end{theorem}

By using the two theorems above, we can easily deduce the lower bound of our next result. However, an upper bound for $\gamma_{t,coi}(G)$ in terms of the vertex cover number can also be deduced. We first consider the case whether $G$ is a star graph $S_n$ for which is known that $\gamma_{t,coi}(S_n)=2$ and $\alpha(S_n)=1$.

\begin{remark}
For any star graph $S_n$, $\gamma_{t,coi}(S_n)=2=2\alpha(S_n).$
\end{remark}

In concordance with the remark above, for our next result we exclude the case of star graphs and see that they behave in a different manner.

\begin{theorem}\label{th-cover-tcid}
For any graph $G$ of order $n$ different from a star graph,
$$\alpha(G)\le \gamma_{t,coi}(G)\le 2\alpha(G)-1.$$
\end{theorem}

\begin{proof}
The lower bound follows from Theorems \ref{th gallai} and \ref{th-indep-tcid}. If $\alpha(G)\ge n/2$, then $\gamma_{t,coi}(G)\le n-1\le 2(n/2)-1\le 2\alpha(G)-1$. Thus, from now on in this proof we consider $\alpha(G) < n/2$. Now, let $C$ be an $\alpha(G)$-set.

We choose two vertices $u,v\in C$ with the minimum possible distance between them and let $P$ be a shortest $u-v$ path. Clearly, $V(P)\cap C=\{u,v\}$ and the distance between $u$ and $v$ is one or two (notice this also means $2\le |V(P)|\le 3$). For each vertex $x\in C -\{u, v\}$, choose
a neighbor $x'$ of $x$. Then $C\cup V(P)\cup \{x' : x\in C-\{u, v\}\}$ is a TC-ID set of cardinality $2|C| + |V(P)|- 4\le 2\alpha(G)- 1$, which completes the proof of the upper bound.
\end{proof}

The bounds above are tight. For instance, a characterization of that trees achieving the equality in the lower bound was given in \cite{Cabrera2017} (note that in \cite{Cabrera2017} the trees $T$ of order $n$ satisfying equality in the bound $\gamma_{t,coi}(T)=n-\beta(T)$ were characterized, which equals the lower bound of Theorem \ref{th-cover-tcid}, in concordance with Theorem \ref{th gallai}). The upper bound is attained for an infinite family of graphs, as we next show. To this end, we need the following operations for edges or induced paths $P_3$ of a graph $G$.\\
{\bf\em Subdivision:} Given an edge $uv$, remove the edge, add a vertex $w$ and the edges $uw$, $wv$.\\
{\bf\em Inflation of size $k$:} Given an induced path $P_3=uvw$ of $G$, in which $v$ has degree two, remove the vertex $v$ and the two incident edges, and replace them with $k$ vertices $v_1,v_2,\dots,v_k$ and edges $uv_i, v_iw$ for every $i\in \{1,\dots,k\}$.\\
{\bf\em Addition of $t$ pendant vertices:} Given a vertex $x$ add $t$ new vertices $y_1,\dots,y_t$ and the edges $xy_i$ for every $i\in \{1,\dots,t\}$.\\

Now, a graph $H_{n,a,b}\in \mathcal{F}_1$ is a graph obtained from a star graph $S_n$ by making the following sequence of operations, which we will call as \textbf{Sequence I}.
\begin{itemize}
  \item[(a)] Apply the operation ``\emph{Subdivision}'' to $a$ ($1\le a\le n$) edges of $S_n$.
  \item[(b)] Apply the operation ``\emph{Inflation of size $k_i$}'' with $k_i\ge 2$ to $b$ ($0\le b\le a$) paths $P_3^{(i)}$ obtained from (a).
  \item[(c)] Apply the operation ``\emph{Addition of $q_i$ pendant vertices}'', $q_i\ge 0$, to the $b$ vertices corresponding to leaves of $S_n$ obtained in the step (b).
  \item[(d)] Apply the operation ``\emph{Addition of $t_i$ pendant vertices}'', $t_i\ge 1$, to the leaves $v_i$ belonging to the remaining $a-b$ paths obtained from (a), which were not ``inflated'' in (b).
  \item[(e)] If $a=n$ and $b=0$ (notice that in this case $H_{n,a,b}$ is a tree such that the central vertex of the original star graph $S_n$ has no adjacent leaves), then apply the operation ``\emph{Addition of $t$ pendant vertices}'', $t\ge 1$, to the vertex corresponding to the central vertex of $S_n$.
  \item[(f)] If $a=n$ and $b>0$, then apply the operation ``\emph{Addition of $t$ pendant vertices}'', $t\ge 0$, to the vertex corresponding to the central vertex of $S_n$.
\end{itemize}

As an example, to obtain the cycle $C_4$ (which belongs to $\mathcal{F}_1$) we begin with the star $S_1$ (a path $P_2$), next we apply the operation ``\emph{Subdivision}'' to the unique edge of $S_1$ and then we apply the operation ``\emph{Inflation of size $2$}'' to the path $P_3$ obtained in the previous step. Note that different sequences of operations would lead to the same graph. For instance, the graph $P_5$ can be obtained from the star $S_1$ by subdividing its unique edge and then adding a pendant vertex to the leaf corresponding to the subdivision, as well as another pendant vertex to the center of $S_1$ (coincidentally such center is also a leaf). Moreover, the graph $P_5$ is obtained from the star $P_3$ by subdividing one of its edges and then adding a pendant vertex to the leaf corresponding to such subdivision. On the other hand, we remark that three integers $n,a,b$ would produce different graphs $H_{n,a,b}$ depending on the addition of pendant vertices that would be done. However, since it is not significant for our work to denote them, we skip to use the notations for the addition of pendant vertices.  A fairly representative graph of the family $\mathcal{F}_1$ is given in Figure \ref{tree-T-2-2-Q}.

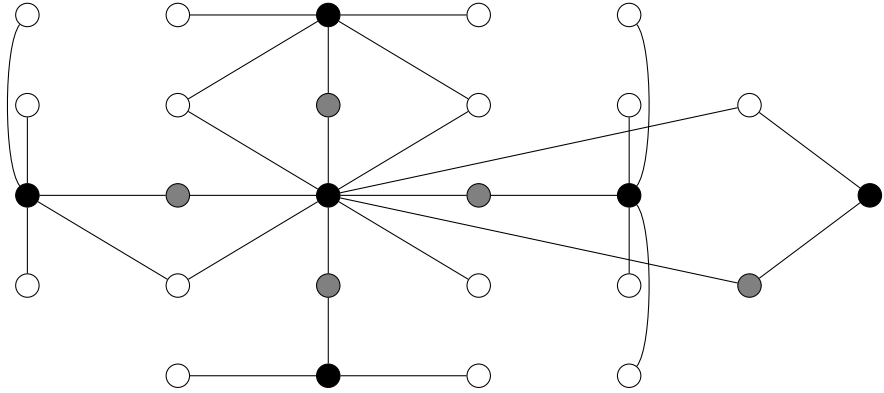
\begin{figure}[h]
\centering
\begin{tikzpicture}[scale=.8, transform shape]

\node [draw, shape=circle, fill=black] (c00) at  (0,0) {};
\node [draw, shape=circle, fill=gray] (a20) at  (2.5,0) {};
\node [draw, shape=circle, fill=black] (a40) at  (5,0) {};
\node [draw, shape=circle, fill=gray] (b20) at  (-2.5,0) {};
\node [draw, shape=circle, fill=black] (b40) at  (-5,0) {};
\node [draw, shape=circle, fill=gray] (c02) at  (0,1.5) {};
\node [draw, shape=circle] (a22) at  (2.5,1.5) {};
\node [draw, shape=circle] (a42) at  (5,1.5) {};
\node [draw, shape=circle] (b22) at  (-2.5,1.5) {};
\node [draw, shape=circle] (b42) at  (-5,1.5) {};
\node [draw, shape=circle, fill=black] (c04) at  (0,3) {};
\node [draw, shape=circle] (a24) at  (2.5,3) {};
\node [draw, shape=circle] (a44) at  (5,3) {};
\node [draw, shape=circle] (b24) at  (-2.5,3) {};
\node [draw, shape=circle] (b44) at  (-5,3) {};
\node [draw, shape=circle, fill=gray] (cc02) at  (0,-1.5) {};
\node [draw, shape=circle] (aa22) at  (2.5,-1.5) {};
\node [draw, shape=circle] (aa42) at  (5,-1.5) {};
\node [draw, shape=circle] (bb22) at  (-2.5,-1.5) {};
\node [draw, shape=circle] (bb42) at  (-5,-1.5) {};
\node [draw, shape=circle, fill=black] (cc04) at  (0,-3) {};
\node [draw, shape=circle] (aa24) at  (2.5,-3) {};
\node [draw, shape=circle] (aa44) at  (5,-3) {};
\node [draw, shape=circle] (bb24) at  (-2.5,-3) {};

\node [draw, shape=circle] (d1) at  (7,1.5) {};
\node [draw, shape=circle, fill=gray] (d2) at  (7,-1.5) {};
\node [draw, shape=circle, fill=black] (d3) at  (9,0) {};

\draw(c00)--(d1)--(d3)--(d2)--(c00);
\draw(c04)--(c02)--(c00)--(cc02)--(cc04);
\draw(b40)--(b20)--(c00)--(a20)--(a40);
\draw(b22)--(c00)--(a22);
\draw(bb22)--(c00)--(aa22);
\draw(b24)--(c04)--(a24);
\draw(bb24)--(cc04)--(aa24);
\draw(b42)--(b40)--(bb42);
\draw(a42)--(a40)--(aa42);
\draw(a22)--(c04)--(b22);
\draw(b40)--(bb22);

\draw (b40) .. controls (-5.4,0.5) and (-5.4,2.5) .. (b44);
\draw (a40) .. controls (5.4,0.5) and (5.4,2.5) .. (a44);
\draw (a40) .. controls (5.4,-0.5) and (5.4,-2.5) .. (aa44);

\end{tikzpicture}
\caption{A graph $H_{5,5,3}\in \mathcal{F}_1$ where the six bolded vertices form an $\alpha(H)$-set and gray vertices form a possible set to be added to the bolded vertices to get a $\gamma_{t,coi}(H)$-set, which has cardinality eleven.}
\label{tree-T-2-2-Q}
\end{figure}

\begin{remark}\label{family-F-values}
For any graph $H_{n,a,b}\in \mathcal{F}_1$, $\alpha(H_{n,a,b})=a+1$ and $\gamma_{t,coi}(H_{n,a,b})=2a+1$.
\end{remark}

\begin{proof}

For any edge of $S_n$ which was subdivided in step (a), it appears either a path $P_4$ or a cycle $C_4$ and all these subgraphs have in common only one vertex (the corresponding one to the center of $S_n$). Thus, in order to cover all the edges of $H_{n,a,b}$, at least $a+1$ vertices are required. Moreover, a set given by those $a$ leaves corresponding to the $a$ edges of the star $S_n$ which were subdivided together with the central vertex form a vertex cover of cardinality $a+1$. Thus, the equality $\alpha(H_{n,a,b})=a+1$ follows.

Now, let $D$ be a  $\gamma_{t,coi}(H_{n,a,b})$-set. We analyze the following situations for every edge $wu$ (assume $w$ is the center of $S_n$) of the star which is initially subdivided.\\

Case 1: There is only one path between $w$ and $u$ in $H_{n,a,b}$. Hence, the edge $wu$ was subdivided with a vertex, say $v$, and not inflated, which made a required addition of at least one pendant vertex, say $u'$, to the leaf $u$. Thus, in order to totally dominate $u'$, $|D\cap \{v,u,u'\}|\ge 2$.\\

Case 2: There are at least two paths between $w$ and $u$ in $H_{n,a,b}$. Clearly, this means $wu$ was subdivided and then inflated with at least two vertices,  say $v_1,\dots,v_r$, $r\ge 2$. Moreover, probably some pendant vertices were added to $u$. So, in order to totally dominate $v_1,\dots,v_r,u$ (and probably other extra leaves adjacent to $u$), at least two vertices of $v_1,\dots,v_r,u$ are required. \\

We next consider the vertex $w$ separately. If $a<n$, then the vertex $w$ has a least one adjacent leaf which needs to be totally dominated. Thus, $w$ must belong to $D$. On the contrary, if $a=n$, then we must consider the value $b$. If $b=0$, then no path $P_3$ was inflated ($H_{n,a,b}$ is a tree) and so, by step (e), $w$ has at least one adjacent leaf which needs to be totally dominated, which means $w$ must belong to $D$ again. Finally, we assume $b>0$. Thus, at least one path $P_3$ was inflated and there is a cycle $C_4^{(j)}$ to which $w$ belongs. Also, it may happen $w$ has no adjacent leaves. Now, note that if $w\notin D$, then the two vertices of $C_4^{(j)}$ adjacent to $w$ must belong to $D$, since $\overline{D}$ is an independent set. Moreover, the fourth vertex of $C_4^{(j)}$ must belong to $D$ too, in order to get the vertices of $D$ totally dominated. As a consequence, at least three vertices of the cycle are in $D$, which is equivalent to have in $D$ the vertex $w$, one of its neighbors in $C_4^{(j)}$ and the vertex of $C_4^{(j)}$ which is not adjacent to $w$.

Consequently, we can deduce that for any set of vertices of a subgraph of $H_{n,a,b}$, induced by the vertices obtained in a subdivision of one of the $a$ leaves of $S_n$ and probably the corresponding addition of some pendant vertices, at least two of these vertices are in $D$. Moreover, one extra vertex is required, which could mainly be the central vertex $w$ of $S_n$. Thus,  $\gamma_{t,coi}(H_{n,a,b})=|D|\ge 2a+1$.

On the other hand, by using Theorem \ref{th-cover-tcid}, we obtain that $\gamma_{t,coi}(H_{n,a,b})\le 2\alpha(H_{n,a,b})-1=2a+1$ and the equality follows for $\gamma_{t,coi}(H_{n,a,b})$.
\end{proof}

Now, a graph $H\in \mathcal{F}_2$ is a graph obtained from the cycle $C_6$ by making the following sequence of operations, which we will call as \textbf{Sequence II}.
\begin{itemize}
  \item[(a)] Apply the operation ``\emph{Addition of $t_i$ pendant vertices}'', $t_i\ge 0$ and $i\in \{1,2,3\}$, to the three vertices, say $v_1,v_2,v_3$, of a $\beta(C_6)$-set, respectively.
  \item[(b)] If there is a $t_i=0$ from the above operation and the degree of $v_i$ is two, then apply the operation ``\emph{Inflation of size $k$}'' with $k\ge 2$ to one of the two possible paths of order three between $v_i$ and the other two vertices in $\{v_1,v_2,v_3\}-\{v_i\}$.
  \item[(c)] Apply the operation ``\emph{Inflation of size $k_i$}'' with $k_i\ge 1$ and $i\in \{1,2,3\}$ to the three possible paths of order three between $v_1,v_2,v_3$.
\end{itemize}

An example of a graph of the family $\mathcal{F}_2$ appears in Figure \ref{Figure-family-F-2}.

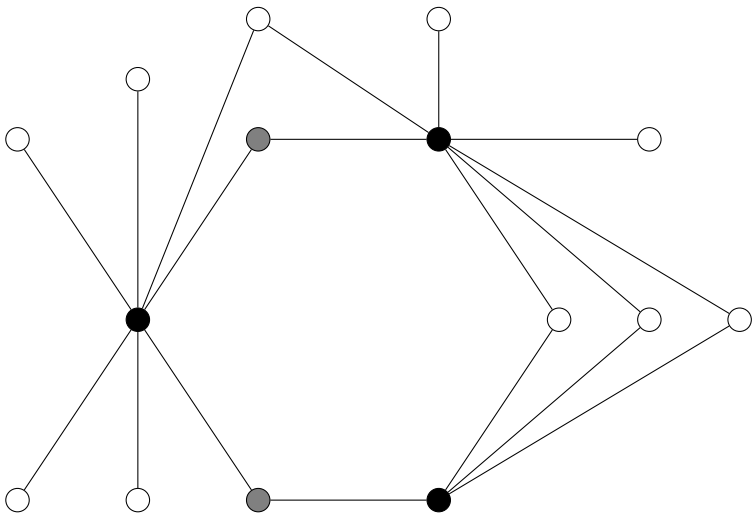
\begin{figure}[h]
\centering
\begin{tikzpicture}[scale=.8, transform shape]

\node [draw, shape=circle] (a1) at  (0,-3) {};
\node [draw, shape=circle, fill=black] (a2) at  (0,0) {};
\node [draw, shape=circle, fill=gray] (a3) at  (2,-3) {};
\node [draw, shape=circle, fill=gray] (a4) at  (2,3) {};
\node [draw, shape=circle] (a5) at  (2,5) {};
\node [draw, shape=circle, fill=black] (a6) at  (5,-3) {};
\node [draw, shape=circle, fill=black] (a7) at  (5,3) {};
\node [draw, shape=circle] (a8) at  (5,5) {};
\node [draw, shape=circle] (a9) at  (7,0) {};
\node [draw, shape=circle] (a10) at  (8.5,0) {};
\node [draw, shape=circle] (a11) at  (8.5,3) {};
\node [draw, shape=circle] (a12) at  (10,0) {};
\node [draw, shape=circle] (a13) at  (0,4) {};
\node [draw, shape=circle] (b) at  (-2,-3) {};
\node [draw, shape=circle] (c) at  (-2,3) {};

\draw(a1)--(a2)--(a4)--(a7)--(a9)--(a6)--(a10)--(a7)--(a12)--(a6)--(a3)--(a2);
\draw(a2)--(a5)--(a7)--(a8);
\draw(a7)--(a11);
\draw(a2)--(a13);
\draw(b)--(a2)--(c);
\end{tikzpicture}
\caption{A graph $H\in \mathcal{F}_2$ where the three bolded vertices form an $\alpha(H)$-set and the two gray vertices form a possible set to be added to the bolded vertices to get a $\gamma_{t,coi}(H)$-set, which has cardinality five.}
\label{Figure-family-F-2}
\end{figure}

The following result concerning the values of $\alpha(H)$ and $\gamma_{t,coi}(H)$ for graphs $H\in \mathcal{F}_2$ is straightforward to observe.

\begin{remark}\label{family-F-2-values}
For any graph $H\in \mathcal{F}_2$, $\alpha(H)=3$ and $\gamma_{t,coi}(H)=5$.
\end{remark}

According to the Remarks above, we can easily check that the upper bound of Theorem \ref{th-cover-tcid} is achieved for any graph $G\in \mathcal{F}_1\cup \mathcal{F}_2$. Moreover, we next prove that precisely the graphs of these families are the only ones achieving the upper bound of Theorem \ref{th-cover-tcid}. To this end, we need the following two lemmas whose proofs can be made by using some similar techniques as in the proof of Theorem \ref{th-cover-tcid}.

\begin{lemma}\label{no-independent}
If a graph $G$ contains an $\alpha(G)$-set which is not independent, then $\gamma_{t,coi}(G)\le 2\alpha(G)-2$.
\end{lemma}

\begin{proof}
Let $C$ be an $\alpha(G)$-set which is not independent. We choose two adjacent vertices $u,v\in C$. For each vertex $x\in C -\{u, v\}$, choose
a neighbor $x'$ of $x$. Then $C\cup \{x' : x\in C-\{u, v\}\}$ is a TC-ID set of cardinality $2(|C|-2)+2\le 2\alpha(G)- 2$.
\end{proof}

\begin{lemma}\label{no-caminos-disj}
If an $\alpha(G)$-set of a graph $G$ contains four different vertices $u_1,u_2,v_1,v_2$ such that a shortest $u_1-u_2$ path and a shortest $v_1-v_2$ path have length two and are vertex disjoint, then $\gamma_{t,coi}(G)\le 2\alpha(G)-2$.
\end{lemma}

\begin{proof}
Let $C$ be an $\alpha(G)$-set such that $u_1,u_2,v_1,v_2\in C$. For each vertex $x\in C -\{u_1,u_2,v_1,v_2\}$, choose
a neighbor $x'$ of $x$ and let $w_1,w_2\notin C$ be two vertices such that $w_1\in N(u_1)\cap N(u_2)$ and $w_2\in N(v_1)\cap N(v_2)$, which exist by assumption. Then $C\cup \{w_1,w_2\} \cup \{x' : x\in C-\{u_1,u_2,v_1,v_2\}\}$ is a TC-ID set of cardinality $2(|C|-4)+6\le 2\alpha(G)- 2$.
\end{proof}

\begin{theorem}
Let $G$ be a graph of order $n$ such that $2\alpha(G)\le n$. Then $\gamma_{t,coi}(G)= 2\alpha(G)-1$ if and only if $G\in \mathcal{F}_1\cup \mathcal{F}_2$.
\end{theorem}

\begin{proof}

In one hand, if $G\in \mathcal{F}_1\cup \mathcal{F}_2$, then it clearly happens that $\gamma_{t,coi}(G)= 2\alpha(G)-1$ according to Remarks \ref{family-F-values} and \ref{family-F-2-values}.

On the second hand, assume $\gamma_{t,coi}(G)= 2\alpha(G)-1$ and let $D$ be any $\alpha(G)$-set. We first notice that $D$ must induce an independent set according to Lemma \ref{no-independent}. We shall now proceed by proving some partial claims that will further give our required conclusion.\\

\noindent{\bf Claim 1:} $G$ has no triangles (cycles of order three).

\noindent{\bf Proof of Claim 1:} If there is a triangle, then, in order to cover all its edges, at least two of its vertices must belong to $D$. So, this cover is not an independent set. Thus we get a contradiction by using Lemma \ref{no-independent}. {\small $(\Box)$}\\

\noindent{\bf Claim 2:} $G$ has no induced cycles of order five or larger than six.

\noindent{\bf Proof of Claim 2:} Suppose $G$ contains a cycle $C_r$ with $r=5$ or $r\ge 7$. In order to cover all the edges of $C_r$, and since $r\ne 6$, there must be two vertices in $D\cap V(C_r)$ at distance one (which means $D$ is not independent), or there are four different vertices $u_1,u_2,v_1,v_2\in D\cap V(C_r)$ such that a shortest $u_1-u_2$ path and a shortest $v_1-v_2$ path have length two and are vertex disjoint. Thus, we obtain contradictions by using Lemmas \ref{no-independent} and \ref{no-caminos-disj}. {\small $(\Box)$}\\

As a consequence of the Claims above, we have that $G$ can only contain cycles of order four or six. We first analyze the case in which $G$ contains a cycle of order six. Let $V(C_6) = \{v_1,\dots,v_6\}$ where $v_1\sim v_2\sim \dots \sim v_6\sim v_1$ ($u\sim v$ means $u,v$ are adjacent). According to Lemma \ref{no-independent}, it must happen $D\cap V(C_6)$ is independent. Thus, without loss of generality we assume $D\cap V(C_6)=\{v_1,v_3,v_5\}$. We consider now several situations.

Suppose $v_2$ has degree larger than two. If $v_2\sim v_j$ with $j\in \{4,6\}$, then $G$ has a triangle, which is not possible. If $v_2\sim v_5$, then for each vertex $x\in C -\{v_1,v_3,v_5\}$, choose a neighbor $x'$ of $x$ and we observe that the set $D\cup \{v_2\}\cup \{x' : x\in D-\{v_1,v_3,v_5\}$ is a TC-ID set of cardinality $2(|D|-3)+4\le 2\alpha(G)-2$, a contradiction. Thus, $v_2$ has a neighbor $z\notin V(C_6)$. Since $v_2\notin D$, it must happen $z\in D$. Since $z\not\sim v_1$ and $z\not\sim v_3$ (otherwise there would be a triangle), we obtain a contradiction with Lemma \ref{no-caminos-disj} by using the vertices $z,v_1,v_3,v_5$. As a consequence, $v_2$ must have degree two, and by symmetry also $v_4,v_6$ are of degree two too.

Suppose $v_1$ has degree two. Hence, for each vertex $x\in C -\{v_1,v_3,v_5\}$, we choose a neighbor $x'$ of $x$ and observe that the set $(D-\{v_1\})\cup \{v_2,v_6\}\cup \{x' : x\in D-\{v_1,v_3,v_5\}$ is a TC-ID set of cardinality $2(|D|-3)+4\le 2\alpha(G)-2$, a contradiction. So, $v_1$ must have degree at least three and, by symmetry also $v_3,v_5$ are of degree at least three too.

We consider now a vertex $x\in N(v_1)-\{v_2,v_6\}$. Notice that $x\ne v_3,v_5$ (otherwise there would be a triangle). Also, $x\ne v_4$, by using the same idea as before whether $v_2$ has degree larger than two and $v_2\sim v_5$. Suppose $x$ has degree larger than one and let $x'\in N(x)-\{v_1\}$. Since $D$ is independent and the edge $xx'$ must be covered by $D$, $x'\in D$. If $x'\ne v_3$ and $x'\ne v_5$, then we obtain a contradiction with Lemma \ref{no-caminos-disj} by using the vertices $x',v_1,v_3,v_5$ (notice that $x'\not\sim v_1$ since $D$ is independent). As a consequence, we obtain that any neighbor $x$ of $v_1$ is either of degree one or has a neighbor in $V(C_6)-\{v_1\}$.

We next consider the latter situation whether $x'\in V(C_6)-\{v_1\}$. Clearly $x'\ne v_2,v_4,v_6$. Suppose $x$ is neighbor of $v_3$ and of $v_5$. We choose a neighbor $y'$ of $y$ for every $y\in D-\{v_1,v_3,v_5\}$ and observe that the set $D\cup \{x\}\cup \{y' : y\in D-\{v_1,v_3,v_5\}$ is a TC-ID set of cardinality $2(|D|-3)+4\le 2\alpha(G)-2$, a contradiction. Thus, $x$ is a neighbor of either $v_3$ or $v_5$, in which case, it happens $x$ has degree two. By symmetry, we obtain similar conclusions for $v_3$ and $v_5$. That is, for any $v_i$ with $i\in \{1,3,5\}$, $N(v_i)$ is given by leaves or vertices of degree two. In the latter case, if $x\in N(v_i)-V(C_6)$, then $N(x)=\{v_i,y\}$ where $y\in D\cap V(C_6)-\{v_i\}$.

As a consequence, we observe that $G$ can be obtained from a cycle $C_6$ by the Sequence II of operations described above, or equivalently $G\in \mathcal{F}_2$. We may now consider the case in which $G$ contains a cycle $C_4$, but $G$ does not contain the cycle $C_6$.\\

\noindent{\bf Claim 3:} $G$ does not contain vertex disjoint cycles.

\noindent{\bf Proof of Claim 3:} We directly obtain a contradiction by Lemma \ref{no-caminos-disj}, since in this case there are four different vertices $u_1,u_2,v_1,v_2$ (two of them in one cycle, the other two in the other cycle) such that a shortest $u_1-u_2$ path and a shortest $v_1-v_2$ path have length two and are vertex disjoint. {\small $(\Box)$}\\

Thus, if $G$ contains more than one cycle $C_4$, then they are not vertex disjoint. Moreover, we can next see that not two adjacent vertices of a cycle can be in any other cycle.\\

\noindent{\bf Claim 4:} If two cycles $C_4$ of $G$ has exactly two vertices in common, then these vertices are not adjacent.

\noindent{\bf Proof of Claim 4:} Suppose there are two cycles $C_4$ having two adjacent vertices in common. Assume the cycles are $C_4^{(1)}=v_1v_2v_3v_4v_1$ and $C_4^{(2)}=v_1v_2v_5v_6v_1$. Hence, we note that exactly three vertices of $\{v_1,\dots,v_6\}$ must belong to $D$, otherwise there are two adjacent vertices in $D$. Indeed, such vertices are either $v_1,v_3,v_5$ or $v_2,v_4,v_6$, say for instance $v_1,v_3,v_5$. We choose a neighbor $x'$ of $x$ for every $x\in D-\{v_1,v_3,v_5\}$ and observe that the set $D\cup \{v_2\}\cup \{x' : x\in D-\{v_1,v_3,v_5\}$ is a TC-ID set of cardinality $2(|D|-3)+4\le 2\alpha(G)-2$, which is a contradiction. {\small $(\Box)$}\\

Now, according to the Claims above, if $G$ contains more than one cycle $C_4$, then only the following situations can occur.
\begin{itemize}
\item Any two cycles have exactly one vertex in common.
\item Any two cycles have exactly two vertices in common which are not adjacent.
\item Any two cycles have exactly three vertices in common.
\end{itemize}
We note that the situation in which two cycles of $G$ have exactly three vertices in common can be understood as $G$ has three cycles with two vertices in common. We now turn our attention on the following.\\

\noindent{\bf Claim 5:} There is a vertex $w\in D$ such that $d(w,x)=2$ for every $x\in D-\{w\}$.

\noindent{\bf Proof of Claim 5:} We first note that there are at least two vertices $w,x\in D$ such that $d(x,w)=2$, otherwise there would be an edge not covered by $D$. Let $h$ be a vertex adjacent to $w$ and $x$. Suppose there is a vertex $y\in D$ such that $d(w,y)\ne 2$ and $d(x,y)\ne 2$ (note that $d(w,y)\ne 1$ and $d(x,y)\ne 1$). Thus, since there are no cycles of order larger than four in $G$, there must happen one of the following situations.

\begin{itemize}
\item[(a)] \textbf{There is a shortest path joining $y$ and $h$ not containing $w$ nor $x$.} Also, $y$ is different from the neighbor of $h$, say $h'$, in such path. In such case, in order to cover the edge $hh'$, it must happen $h'\in D$. Thus, we obtain a contradiction by using Lemma \ref{no-caminos-disj} and the vertices $h',w,x,z$ where $z$ is a vertex at distance two from $x$ in the $x-y$ path.
\item[(b)] \textbf{Without loss of generality, there is a shortest path joining $y$ and $x$ containing $w$.} Thus, there must be a vertex $y'\in D$ belonging to this path such that $d(y,y')=2$ (it cannot be $d(y,y')=1$ since $D$ is independent), otherwise there should be a not covered edge. Clearly $w\ne y'$. Thus, we obtain a contradiction by using Lemma \ref{no-caminos-disj} and the vertices $y,y',w,x$.
\end{itemize}
As a consequence, the vertex $y$ has distance two to $x$ or to $w$. Moreover, if $d(w,y)=2$ and $d(x,y)=2$, then we there is a cycle of order six, which is not possible. Thus, $y$ has distance two to exactly one vertex of $x$ and $w$. From now on, we assume $d(y,w)=2$.

We next prove that for any vertex $z\in D-\{x,y,w\}$, it follows $d(z,w)=2$ too. If $D=\{x,y,w\}$, then we are done. So, me may suppose there is a vertex $z'\in D-\{x,y,w\}$ such that $d(z',w)\ne 2$ (clearly $d(z',w)>2$). Consider now the shortest path between $z'$ and $w$, say $z' z'_1 z'_2 \dots z'_q w$. Notice that $q\ge 2$. In order to cover the edge $z'_1z'_2$, it must be $z'_2\in D$. So, we obtain a contradiction by using Lemma \ref{no-caminos-disj} and the vertices $x,w,z',z'_2$. Therefore, for any vertex $z\in D-\{w\}$, we obtain that $d(z,w)=2$ and the claim is proved. {\small $(\Box)$}\\

Next step gives some result on the distances between any two vertices $x,y\in D-\{w\}$.\\

\noindent{\bf Claim 6:} For any two vertices $x,y\in D-\{w\}$, any shortest path between $x$ and $y$ passes through $w$.

\noindent{\bf Proof of Claim 6:} From Claim 5, we know that $d(x,w)=d(y,w)=2$. Thus $d(x,y)\le 4$. Clearly $d(x,y)>1$, since $x,y$ cannot be adjacent. Let $x',y'\in N(w)$ such that $x'\in N(x)$ and $y'\in N(y)$. If $x\sim y'$ or $y\sim x'$ (say $x\sim y'$), then we choose a neighbor $z'$ of $z$ for every $z\in D-\{w,x,y\}$ and observe that the set $D\cup \{y'\}\cup \{z' : z\in D-\{w,x,y\}$ is a TC-ID set of cardinality $2(|D|-3)+4\le 2\alpha(G)-2$, which is a contradiction. Thus, neither $x\sim y'$ nor $y\sim x'$. If there is a vertex $z\in N(x)\cap N(y)$, then $w x' x z y y' w$ is a cycle $C_6$ in $G$, which is not possible. Thus $d(x,y)\ne 2$. By using a similar reasoning, it can be deduced that $d(x,y)\ne 3$ and so, $d(x,y)=4$. If there is another path of length four between $x$ and $y$ not containing $w$, then we have one of the following situations.
\begin{itemize}
  \item There is a vertex $w'\in D$ such that $x',y'\in N(w')$ (note that $w'$ must be in $D$ in order to cover the edges $w'y'$, $w'x'$). In such case, we obtain a contradiction by using Lemma \ref{no-caminos-disj} and the vertices $x,w,w',y$.
  \item There are three vertices $x_1,y_1,w''\ne x',y',w$ such that $x_1\in N(x)$, $y_1\in N(y)$ and $x_1,y_1\in N(w'')$. In such situation, $w x' x x_1 w'' y_1 y y' w$ is an induced cycle of order eight in $G$, which is not possible.
  \item Similarly to the case above, if either $x_1=x'$ or $y_1=y'$, then we obtain an induced cycle of order six in $G$, which is also not possible.
\end{itemize}
Therefore, any shortest path between $x$ and $y$ passes throughout $w$. {\small $(\Box)$}\\

We now give several facts which are consequences of the Claims above, in order to deduce the structure of the graph $G$.
\begin{itemize}
  \item The set $V(G)-D$ is independent (otherwise there is an edge not covered by $D$).
  \item If $x,y\in D-\{w\}$, then $N(x)\cap N(y)=\emptyset$.
  \item If $z\in N(x)$ for some $x\in D-\{w\}$, then either $z\in N(w)$ and $z$ has degree two, or $z$ is a vertex of degree one.
  \item If $z'\in N(w)$ is not a vertex of degree one, then there is exactly one vertex $x\in D$ such that $N(z')=\{w,x\}$ (equivalently $z'$ has degree two).
\end{itemize}
As a consequence of the items above, as well as from the Claims, and all the reasoning till this point, we observe that $D$ is formed by $w$ and a set of vertices $v_1,v_2\dots,v_r$ (satisfying the properties above). Clearly, for any vertex $v_i$, the set of its neighbors are either leaves or vertices of degree two adjacent to $w$. Moreover, if $v_i$ has only one neighbor of degree two, then it must have at least one adjacent leaf (otherwise one can find a cover set of smaller cardinality). In this sense, such set of vertices can clearly be obtained from a leaf of a star by making a subdivision of the corresponding edge, an inflation of the path $P_3$ obtained from the subdivision and a subsequent addition of some pendant vertices. On the other hand, if $w$ has some adjacent leaves, then they could be obtained directly from a star, if subdivisions were not done to all the leaves of the star or, by a subsequent addition of leaves to the center of the original star, if all its leaves would have been subdivided. Therefore, it is then concluded that the graph $G$ was obtained from a star by making the Sequence I of operations previously described, which means $G\in \mathcal{F}_1$ and the proof is completed.
\end{proof}

We close this section with two bounds for $\gamma_{t,coi}(G)$ in terms of order, size and minimum and maximum degrees.

\begin{proposition}
Let $G$ be a graph of order $n$, minimum and maximum degrees $\delta$ and $\Delta$, respectively. Then $\gamma_{t,coi}(G)\geq \frac{n\delta}{\Delta+\delta-1}$.
\end{proposition}

\begin{proof}
Let $D$ be a $\gamma_{t,coi}(G)$-set. Hence, the subgraph induced by $V(G)-D$ is edgeless. So, $(n-|D|)\delta=(|V(G)-D|)\delta \leq E(V(G)-D, D)\leq |D|(\Delta-1) $. Furthermore, it follows that $\gamma_{t,coi}(G)\geq \frac{n\delta}{\Delta+\delta-1}$.
\end{proof}

\begin{proposition}
Let $G$ be a graph of order $n$, size $m$, minimum and maximum degrees $\delta$ and $\Delta, respectively$. Then $\gamma_{t,coi}(G) \geq \frac{2m+n\delta}{3\Delta+\delta-2}$.
\end{proposition}

\begin{proof}
Let $D$ be a $\gamma_{t,coi}(G)$-set. Hence, the subgraph induced by $V(G)-D$ is edgeless. So, $E(V(G)-D,D) + E(D,D) = m$. Now, notice that $  E(V(G)-D,D)\leq  |D|(\Delta-1)$ and $ E(D,D)\leq \frac{|D|\Delta-(n-|D|)\delta}{2} $. Adding this inequations, we have $m= E(V(G)-D,D) + E(D,D)\leq \frac{|D|\Delta-(n-|D|)\delta}{2} + |D|(\Delta-1) $. Therefore, it follows that $\gamma_{t,coi}(G) \geq \frac{2m+n\delta}{3\Delta+\delta-2}$.
\end{proof}

The two bounds above are attained for instance for the double stars $S_{k,k}$ (each non leaf vertex is adjacent to $k$ leaves), which has order $2(k+1)$, size $m=2k+1$, minimum degree $\delta=1$, maximum degree $\Delta=k+1$  and $\gamma_{t,coi}(S_{k,k})=2$.

\section{The case of trees}

In order to easily proceed with our exposition, and based on the following known bound, from now on we say that a tree $T$ belongs to the family $\mathcal{T}_{\gamma_t}$, if $\gamma_{t,coi}(T)=\gamma_t(T)$. Moreover, we assume in this section that $|S(T)|\geq 2$, since the case $|S(T)|=0$ ($T$ is a $P_2$ and $\gamma_{t,coi}(T)$ is not defined) and  $|S(T)|=1$ ($T$ is a star graph $S_n$ and $\gamma_{t,coi}(T)=2$) are straightforward to study.

\begin{theorem}{\em \cite{Soner2012}}\label{min-bound}
For any graph $G$, $\gamma_{t,coi}(G)\geq \gamma_t(G)$.
\end{theorem}

It is now our goal to characterize the family of trees achieving the equality in the bound above. To this end, we observe the following basic results, which can easily be obtained by using some known properties of minimum total dominating sets.

\begin{proposition}{\em \cite{Cockayne1980}}\label{minimal-TDS}
If $S$ is a minimal total dominating set of a connected graph $G=(V,E)$, then each $v\in S$ has at least one of the following two properties.
\begin{itemize}
  \item[\em{(i)}] There exists a vertex $w\in V - S$ such that $ N(w)\cap S=\{v\}$.
  \item[\em{(ii)}] The subgraph induced by $S-\{v\} $ contains an isolated vertex.
\end{itemize}
\end{proposition}

The next remark is one useful consequence of the proposition above.

\begin{remark}\label{theo-minimal}
Let $D$ be a $\gamma_{t,coi}(T)$-set of cardinality $\gamma_t(T)$.  Then, for every $v\in D$, at least one of the following conditions is satisfied.
\begin{itemize}
  \item[\em{(i)}] There exists a vertex $u \in D$ such that $ N(u)\cap D=\{v\}$.
  \item[\em{(ii)}] There exists a vertex $w\in V - D$ such that $ N(w)\cap D=\{v\}$.
\end{itemize}
\end{remark}

We may recall to notice that condition $(ii)$ implies that vertex $v$ is a support, because the set $\overline{D}$ is independent.

\begin{lemma}\label{dist-V-D}
Let $T\in \mathcal{T}_{\gamma_t}$ and let $D$ be a $\gamma_{t,coi}(T)$-set containing no leaves. Then for every $v\in V(T)-(D\cup L(T))$ there exist a leaf $h$ such that $d(v,h)\leq 3$.
\end{lemma}

\begin{proof}
Let $v\in V(T)-(D\cup L(T))$. Since $|N(v)|\geq 2$, we consider $N(v)=\{v_1,v_2, \ldots , v_r\}$ with $r\geq 2$. Clearly, $N(v)\subset D$ since $\overline{D}$ is independent. For every $v_i$, with $i\in \{1,\dots, r\}$, by Remark \ref{theo-minimal}, $v_i$ is adjacent to a leaf or there exist a vertex $s_i\in D$ such that $N(s_i)\cap D=\{v_i\}$. Hence, as $s_i\in D$, $N(s_i)\subset V(T)-D$. We assume that for every $i\in \{1,\dots, r\}$, $v_i$ is not  adjacent to a leaf $h$, otherwise  $d(v,h)=2$. Now, we suppose that $(N(s_i)-\{v_i\})\cap L(T)=\emptyset$. Also note that, by condition above, the vertices belonging to $N(s_i)$ are totally dominated by other vertices of $D$. So, we observe that the set $(D-\{s_1,s_2,\ldots , s_r\})\cup \{v\}$ is a total dominating set of $T$ of cardinality smaller than $|D|$, a contradiction. Furthermore, there exist $i\in \{1,\dots, r\}$ such that $(N(s_i)-\{v_i\})\cap L(T)\neq \emptyset$. Thus, for any $h\in (N(s_i)-\{v_i\})\cap L(T)$, it follows $d(v,h)=3$, and this completes the proof.
\end{proof}

From this point, the set of leaves having distance three with respect to at least one other leaf is denoted by $L_3(T)$, and given a $\gamma_{t,coi}(T)$-set $D$, we denote by $V_{2,3}(T)\subset V(T)-D$ the set of vertices having distance two or three to some leaf and by $V_6(T)\subset V(T)-D$ the set of vertices having distance three to some vertex of $V_{2,3}(T)$.

In order to provide a constructive characterization of the trees belonging to the family $\mathcal{T}_{\gamma_t}$, we need the following five operations $F_1$, $F_2$, $F_3$, $F_4$ and $F_5$ on a tree $T$ (by attaching a path $P$ to a vertex $v$ of $T$ we mean adding the path $P$ and joining $v$ to a vertex of $P$). Moreover, through all the next results we make use of the fact that any tree $T$ always contains a $\gamma_{t,coi}(T)$-set which does not contain leaves.

\begin{description}
  \item[Operation $F_1$:]  Attach a path $P_1$ to a vertex of $T$, which is in some $\gamma_{t,coi}(T)$-set.
  \item[Operation $F_2$:]  Attach a path $P_1$ to a vertex of $T$, which is in $L_3(T)$.
  \item[Operation $F_3$:]  Attach a path $P_2$ to a vertex of $T$, which is in $L_3(T)$.
  \item[Operation $F_4$:]  Attach a path $P_3$ to a vertex of $T$, which is in $V_{2,3}(T)$.
  \item[Operation $F_5$:]  Attach a path $P_3$ to a vertex of $T$, which is in $V_6(T)$.
\end{description}

Let $\mathcal{F}$ be the family of trees defined as $\mathcal{F} = \{T \mid T $ is obtained from $P_4$ by a finite sequence of operations $F_1, F_2, F_3, F_4$ or $F_5\}$. The Figure \ref{tree} contains a fairly representative example of a tree $T\in\mathcal{F}$.
We first show  that every tree of the family $\mathcal{F}$ belongs to the family $\mathcal{T}_{\gamma_t}$.

\begin{figure}[h]
\centering
\begin{tikzpicture}[scale=.8, transform shape]

\node [draw, shape=circle] (v1) at  (-8,0) {};
\node at (-7.5,-0.5) {\Large $v_1$};
\node [draw, shape=circle] (v2) at  (-6,0) {};
\node at (-5.5,-0.5) {\Large $v_2$};
\node [draw, shape=circle] (v3) at  (-4,0) {};
\node at (-3.5,-0.5) {\Large $v_3$};
\node [draw, shape=circle] (v4) at  (-2,0) {};
\node at (-1.5,-0.5) {\Large $v_4$};
\node [draw, shape=circle] (v5) at  (0,0) {};
\node at (0.5,-0.5) {\Large $v_5$};
\node [draw, shape=circle] (v6) at  (2,0) {};
\node at (2.5,-0.5) {\Large $v_6$};
\node [draw, shape=circle] (v7) at  (4,0) {};
\node at (4.5,-0.5) {\Large $v_7$};
\node [draw, shape=circle] (v8) at  (6,0) {};
\node at (6.5,-0.5) {\Large $v_8$};
\node [draw, shape=circle] (v9) at  (8,0) {};
\node at (8.5,-0.5) {\Large $v_9$};
\node [draw, shape=circle] (v10) at  (10,0) {};
\node at (10.5,-0.5) {\Large $v_{10}$};

\node [draw, shape=circle] (w1) at  (-4,-2) {};
\node at (-3.5,-2.5) {\Large $w_1$};
\node [draw, shape=circle] (w2) at  (6,-2) {};
\node at (6.5,-2.5) {\Large $w_2$};

\node [draw, shape=circle] (u1) at  (-6,-2) {};
\node at (-5.5,-2.5) {\Large $u_1$};
\node [draw, shape=circle] (u2) at  (-8,-2) {};
\node at (-7.5,-2.5) {\Large $u_2$};
\node [draw, shape=circle] (u3) at  (8,-2) {};
\node at (8.5,-2.5) {\Large $u_3$};

\draw(v1)--(v2)--(v3)--(v4)--(v5)--(v6)--(v7)--(v8)--(v9)--(v10);
\draw(v3)--(w1);
\draw(v8)--(w2);
\draw(u2)--(u1)--(w1);
\draw(u3)--(w2);

\end{tikzpicture}
\caption{A tree $T$ obtained from a path $P_4=v_1\,v_2\,v_3\,v_4$, applying the five operations $F_1$, $F_2$, $F_3$, $F_4$ and $F_5$. Firstly, operations $F_4$ and $F_5$ are applied by adding the path $P_3=v_5\,v_6\,v_7$ to the vertex $v_4$ through the edge $v_4v_5$ and the path $P_3=v_8\,v_9\,v_{10}$ to the vertex $v_7$ through the edge $v_7v_8$. Next, we apply the operation $F_1$ twice by attaching the vertices $w_1$ and $w_2$ to the vertices $v_3$ and $v_8$, respectively. Moreover, we apply the operation $F_2$ by adding the vertex $u_3$ to $w_2$. Finally, we apply the operation $F_3$ by adding the path $P_2=u_1\,u_2$ to the vertex $w_1$ through the edge $w_1u_1$.}
\label{tree}
\end{figure}
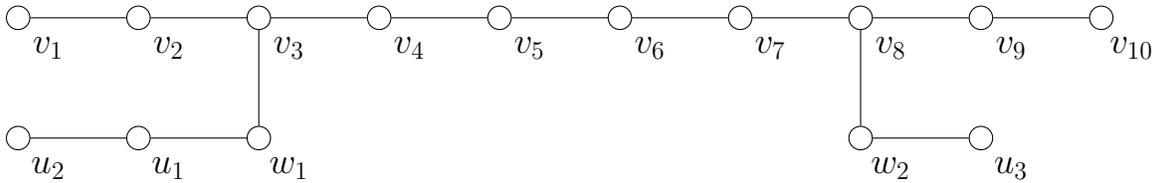

\begin{lemma}\label{lem-right}
If $T \in \mathcal{F}$, then $T\in \mathcal{T}_{\gamma_t}$.
\end{lemma}

\begin{proof}
We proceed by induction on the number $r(T)$ of operations required to construct the tree $T$. If $r(T)=0$, then $T=P_4$ and $T \in \mathcal{T}_{\gamma_t}$.  This establishes the base case. Hence, we now  assume that $k\geq 1$ is an integer and that each tree $T' \in \mathcal{F}$ with $ r(T')<k$ satisfies that $T' \in \mathcal{T}_{\gamma_t}$.
Let $T \in \mathcal{F}$ be a tree for wich $r(T)=k$. Since $T$ can be obtained from a tree $T' \in \mathcal{F}$ with $ r(T')=k-1$ by one of the operations $F_1, F_2, F_3, F_4$ or $F_5$, we shall prove that $T \in \mathcal{T}_{\gamma_t}$, by considering a $\gamma_{t,coi}(T')$-set $D'$ containing no leaves and through the following situations.\\

\textbf{Case 1.} $T$ is obtained from $T'$ by operation $F_1$. Let $u$ be the vertex added to $T$ in order to obtain $T'$. Since $u$ is a leaf of $T$ and is adjacent to a vertex of $D'$, the set $D'$ remains to be a total dominating set in $T$. Moreover, $D'$ is a $\gamma_{t}(T)$-set, since otherwise we would find a total dominating set in $T'$ of cardinality smaller than $\gamma_{t}(T')$. On the other hand, since $(V(T')-D')\cup \{u\}$ is independent, we deduce $D'$ is a TC-ID set in $T$. Thus, $\gamma_{t,coi}(T)\le |D'|=\gamma_{t,coi}(T')=\gamma_t(T')=\gamma_t(T)$ (by also using the inductive hypothesis). Thus, by Theorem \ref{min-bound}, we get the equality $\gamma_{t,coi}(T)=\gamma_t(T)$, which means $T \in \mathcal{T}_{\gamma_t}$.\\

\textbf{Case 2.} $T$ is obtained from $T'$ by operation $F_2$. Assume $T$ is obtained from $T'$ by adding the vertex $u$ and the edge $uv$ where $v\in L_3(T')$. As $v\in L_3(T')$, there exist a path $vu_1u_2h$ in $T'$ where $h$ is a leaf and $u_1, u_2$ are support vertices adjacent to $v,h$, respectively.
Now, in $T$, the vertices $u_2,v$ are supports and belong to any TC-ID set in $T$. Hence, the set $D=D'\cup\{v\}$ is a TC-ID set in $T$, and so
\begin{equation}\label{eq-case-2}
 \gamma_t(T)\leq \gamma_{t,coi}(T)\le \gamma_{t,coi}(T')+1= \gamma_t(T')+1
\end{equation}
(by also using Theorem \ref{min-bound} and the inductive hypothesis). Now, let $A$ be a $\gamma_t(T)$-set containing no leaves. Notice that the vertex $v$ is a support and so, it belong to $A$, also the vertex $u_1$ belongs to $A$ too, because $v$ has degree two. Moreover, note that the set $A-\{v\}$ is a total dominating set in $T'$, which leads to $\gamma_t(T')\le \gamma_t(T)-1$. By using this, it follows that all the inequalities in (\ref{eq-case-2}) must be equalities. Thus $\gamma_{t,coi}(T)=\gamma_t(T)$, and $T \in \mathcal{T}_{\gamma_t}$.\\

\textbf{Case 3.} $T$ is obtained from $T'$ by operation $F_3$. Assume $T$ is obtained from $T'$ by adding the path $P_2=h_1h_2$ to a vertex $v\in L_3(T')$ through the edge $vh_1$. By using some similar reasons as in the case above (now we must use $D=D'\cup \{v,h_1\}$ instead of $D=D'\cup\{v\}$), it is observed that $T \in \mathcal{T}_{\gamma_t}$.\\

\textbf{Case 4.} $T$ is obtained from $T'$ by operation $F_4$. Assume $T$ is obtained from $T'$ by adding the path $P_3=h_1u_1h_2$ to a vertex $v\in V_{2,3}(T')$ through the edge $vh_1$. We notice that $u_1,h_1$ belong to any TC-ID set containing no leaves of $T$. Hence, the set $D=D'\cup\{u_1,h_1\}$ is a TC-ID set in $T$. Thus $\gamma_t(T)\leq \gamma_{t,coi}(T)\le \gamma_{t,coi}(T')+2= \gamma_t(T')+2$ (by also using Theorem \ref{min-bound} and the inductive hypothesis). Now, let $A$ be a $\gamma_t(T)$-set. Since the vertex $u_1$ is a support, it belongs to $A$ and so, $|A\cap \{h_1,u_1,h_2\}|\geq 2$. Moreover, note that $|A\cap V(T')|\geq \gamma_t(T')$. Hence, $\gamma_t(T)=|A|\geq \gamma_t(T')+2$. Again, as in Case 2, we deduce $\gamma_{t,coi}(T)=\gamma_t(T)$, which means $T \in \mathcal{T}_{\gamma_t}$.\\

\textbf{Case 5.} $T$ is obtained from $T'$ by operation $F_5$. Assume $T$ is obtained from $T'$ by adding the path $P_3=h_1u_1h_2$ to a vertex $v\in V_6(T')$ through the edge $vh_1$. By using some similar reasons as in the case above, it can be deduced that $\gamma_{t,coi}(T)=\gamma_t(T)$, which gives $T \in \mathcal{T}_{\gamma_t}$.

\end{proof}

We now turn our attention to the opposite direction concerning the lemma above. In this sense, from now on we shall need the following terminology and notation in our results. Given a tree $T$ and a set $S\subset V(T)$, by $T-S$ we denote a tree obtained from $T$ by removing from $T$ all the vertices in $S$ and all its incident edges (if $S=\{v\}$ for some vertex $v$, then we simply write $T-v$). For an integer $r\geq 2$, by $Q_r$ we mean a graph which is obtained from a path $P_{r+2}=v s s_1 s_2 \ldots s_r$ by attaching a path $P_1$ to every vertex of $P_{r+2}-v$. In Figure \ref{figure-1} we show the example of $Q_5$.

\begin{figure}[h]
\centering
\begin{tikzpicture}[scale=.6, transform shape]
\node [draw, shape=circle] (s) at  (0,0) {};
\node at (0.5,-0.5) {\Large $s$};
\node [draw, shape=circle] (s1) at  (0,1.5) {};
\node [draw, shape=circle] (s2) at  (0,-1.5) {};
\node at (0.5,-1.5) {\Large $v$};

\node [draw, shape=circle] (a1) at  (1.5,0) {};
\node at (1.5,-0.5) {\Large $s_1$};
\node [draw, shape=circle] (a11) at  (1.5,1.5) {};

\node [draw, shape=circle] (a2) at  (3,0) {};
\node at (3,-0.5) {\Large $s_2$};
\node [draw, shape=circle] (a21) at  (3,1.5) {};

\node [draw, shape=circle] (a3) at  (4.5,0) {};
\node at (4.5,-0.5) {\Large $s_3$};
\node [draw, shape=circle] (a31) at  (4.5,1.5) {};

\node [draw, shape=circle] (a4) at  (6,0) {};
\node at (6,-0.5) {\Large $s_4$};
\node [draw, shape=circle] (a41) at  (6,1.5) {};

\node [draw, shape=circle] (a5) at  (7.5,0) {};
\node at (7.5,-0.5) {\Large $s_5$};
\node [draw, shape=circle] (a51) at  (7.5,1.5) {};

\draw(s)--(a1)--(a2)--(a3)--(a4)--(a5);
\draw(s)--(s1);
\draw(s)--(s2);
\draw(a1)--(a11);
\draw(a2)--(a21);
\draw(a3)--(a31);
\draw(a4)--(a41);
\draw(a5)--(a51);

\end{tikzpicture}
\caption{The structure of the tree $Q_5$.}
\label{figure-1}
\end{figure}
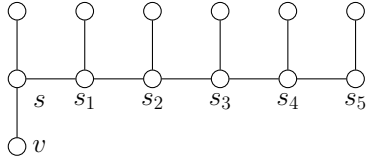

We next show that every tree of the family $\mathcal{T}_{\gamma_t}$ belongs to the family $\mathcal{F}$.

\begin{lemma}\label{lem-left}
If $T\in \mathcal{T}_{\gamma_t}$, then $T \in \mathcal{F}$.
\end{lemma}

\begin{proof}
We proceed by induction on the order $n\geq 4$ of the trees $T\in \mathcal{T}_{\gamma_t}$. If $T$ is a double star, then $T$ can be obtained from $P_4$ by repeatedly applying operation $F_1$. This establishes the base case. We assume next that $k > 4$ is an integer and that each tree $T' \in \mathcal{T}_{\gamma_t}$ with $|V(T')|<k$ satisfies $T' \in \mathcal{F}$.

Let $T$ be a tree such that $T\in \mathcal{T}_{\gamma_t}$ and $|V(T)|=k$. Let $D$ be a $\gamma_{t,coi}(T)$-set containing no leaves and let $B=V(T)-D$. We analyze the following situations.\\

\textbf{Case 1: $|S(T)|<|L(T)|$.} We consider a support vertex $v$ that is adjacent to at least two leaves. Let $h\in N(v)\cap L(T)$ and $T'=T-h$. Thus, the set $D$ is a $\gamma_t(T')$-set too, and by inductive hypothesis, $T' \in \mathcal{F}$. Therefore, since $T$ can be obtained from $T'$ by operation $F_1$, it follows $T \in \mathcal{F}$.\\

\textbf{Case 2: $|S(T)|=|L(T)|$ and $|SS(T)|=0$.} In this case we note that $V(T)=S(T)\cup L(T)$ and clearly, $S(T)$ is a $\gamma_{t,coi}(T)$-set (moreover  $|S(T)|\geq 3$ since otherwise $T$ is a double star). Let $s\in S(T)$ such that $|N(s)\cap S(T)|=1$ (note that such $s$ always exists) and let $h\in L(T)$ be the leaf adjacent to $s$. We first notice that there exists a leaf having distance three to the support $s$. Thus, we deduce that $S'(T)=S(T)-\{s\}$ is a $\gamma_{t,coi}(T')$-set, where $T'=T-h$. By induction hypothesis $T' \in \mathcal{F}$ and, since $T$ can be obtained from $T'$ by operation $F_2$, we get $T \in \mathcal{F}$.\\

\textbf{Case 3: $|S(T)|=|L(T)|$ and $|SS(T)|>0$.}  Herein we denote by $P(x,y)$ the set of vertices of one shortest path between $x$ and $y$, including $x$ and $y$. Let $h,h'$ be two leaves at the maximum possible distance in $T$ such that there is $v \in SS(T)\cap P(h,h')$ with $d(v,h)=2$ or $d(v,h')=2$. Without loss of generality assume that $d(v,h)=2$ and let $s$ be the support adjacent to $h$. Since $|S(T)|=|L(T)|$ and by the maximality of the path between $h$ and $h'$, we observe that $N(s)\subset S(T)\cup \{h,v\}$ and also, that every support vertex is adjacent to exactly one leaf. We have now some possible scenarios.\\

\textbf{Case 3.1  $|N(s)\cap S(T)|=1$.} Hence, by the maximality of the path $P(h,h')$, it must happen that $T$ has an induced subgraph isomorphic to a graph $Q_r$, as previously described, obtained from the vertices $v, s, h$ and some supports, say $s_1,s_2, \ldots s_r\in S(T)$, with the leaves $h_1,h_2,\ldots,h_r$, adjacent to the supports $s_1,s_2, \ldots s_r$, respectively, and such that $\{s_1,\dots,s_r,h_1,\dots,h_r\}\cap P(h,h')=\emptyset$.

Assume $r=1$. Note that $s,s_1\in D$ and that $h,h_1\notin D$. Let $T'=T-h$. Notice that $D$ is also a TC-ID set in $T'$, and so
\begin{equation}\label{eq-T-prime-1}
  \gamma_t(T')\le \gamma_{t,coi}(T')\le \gamma_{t,coi}(T)=\gamma_t(T)
\end{equation}
(by using Theorem \ref{min-bound} and hypothesis). On the other hand, let $A$ be a $\gamma_t(T')$-set containing no leaves. We observe that $s_1 \in A$ because $s_1$ is a support in $T'$, and $s\in A$ because $\delta(s_1)=2$. Thus, clearly $A$ is also a total dominating set in $T$. Hence $\gamma_t(T)\le |A|=\gamma_t(T')$. Thus, all the inequalities in the relation (\ref{eq-T-prime-1}) must be equalities, from which follows $\gamma_{t,coi}(T')=\gamma_t(T')$  and by the inductive hypothesis $T' \in \mathcal{F}$. Since $T$ can be obtained from $T'$ by operation $F_1$, we obtain $T \in \mathcal{F}$.

Assume now $r\geq 2$. Note that $s,s_1,\ldots,s_r\in D$ and that there is a leaf at distance three from $s_r$.  Let $T'=T-h_r$. Hence, $D-\{s_r\}$ is  a TC-ID set in $T'$, and so $\gamma_t(T')\le \gamma_{t,coi}(T')\le \gamma_{t,coi}(T)-1=\gamma_t(T)-1$ (by using Theorem \ref{min-bound} and hypothesis). Moreover, the set $D-\{s_r\}$ is a $\gamma_t(T')$-set, otherwise we would find a total dominating set of $T$ of cardinality smaller than $\gamma_t(T)$, which is not possible.  So, $\gamma_t(T')= \gamma_t(T)-1$ which leads to $\gamma_{t,coi}(T')=\gamma_t(T')$, as in the previous case. Now, by the inductive hypothesis $T' \in \mathcal{F}$, and since $T$ can be obtained from $T'$ by operation $F_2$, we deduce $T\in \mathcal{F}$.\\

\textbf{Case 3.2  $|N(s)\cap S(T)|>1$.} An analogous procedure to the one above (\textbf{Case 3.1}) leads to our desired conclusion, based on the fact that $s$ must have at least two neighbors $s'_1,s''_1\in S(T)$ and there are at least two induced subgraphs isomorphic to the graphs $Q_{r'}$ and $Q_{r''}$,  which can be used instead of $Q_r$ of Case 3.1.\\

\textbf{Case 3.3: $|N(s)\cap S(T)|=0$.} Clearly, $s$ has degree two since it has one leaf neighbor, no support neighbors and cannot have more than one (it has exactly one) semi-support neighbor due to the maximality of $P(h,h')$. Also, it must happen $v\in D$, $h\in B$ and $s\in D$. Assume the subgraph induced by $P(h,h')$ is $h\, s\, v\, u_1\, u_2\, u_3\,u_4 \ldots s'\, h'$, where $h, h' \in L(T)$ and $s,s'\in S(T)$. Note that $N(v)\subset S(T)\cup \{u_1\}$. We consider again some possible scenarios.\\

\textbf{Case 3.3.1: $|N(v)\cap S(T)|>1$.} In this case, the vertex $v$ is also totally dominated by another support $s_v$ different from $s$. Let $h_v$ be the leaf adjacent to the support $s_v$. Notice that $D'=D-\{s\}$ is a TC-ID set of $T'=T-h$. Moreover, we note that the vertex $s$ is a leaf in $T'$ having distance three to the leaf $h_v$. So, by using a similar procedure as above (\textbf{Case 3.1} and $r\geq 2$) we obtain $T' \in \mathcal{F}$. Therefore, due to that $T$ can be obtained from $T'$ by operation $F_2$, it follows $T \in \mathcal{F}$.\\

\textbf{Case 3.3.2: $|N(v)\cap S(T)|=1$ and $|N(u_1)|\geq 3$.} Clearly $s,v$ have degree two and belong to $D$.  We firstly consider the case whether $u_1\in D$. By Remark \ref{theo-minimal} we note that $N(u_1)\cap L(T)\neq \emptyset$ or that there is a vertex $r\in D$ with $N(r)\cap D=\{u_1\}$. If $N(u_1)\cap L(T)\neq \emptyset$, then $D'=D-\{s\}$ is a TC-ID in $T'=T-h$ set, and by using a similar procedure as above (\textbf{Case 3.1} and $r\geq 2$) we obtain $T' \in \mathcal{F}$. Since $T$ can be obtained from $T'$ by operation $F_2$, we are able to claim $T \in \mathcal{F}$.

On the other hand, assume that $N(u_1)\cap L(T)=\emptyset$ and there is a vertex $r\in D$ such that $N(r)\cap D=\{u_1\}$. We note that, by Remark \ref{theo-minimal}, $(N(r)-\{u_1\})\cap L(T)\neq \emptyset$. Hence, $D'=D-\{s\}$ is a TC-ID set in $T'=T-\{h,s\}$. Thus, $\gamma_t(T')\le \gamma_{t,coi}(T')\le \gamma_{t,coi}(T)-1=\gamma_t(T)-1$ (by using Theorem \ref{min-bound} and hypothesis). Again, by using a similar procedure as above (\textbf{Case 3.1} and $r\geq 2$) we get that $\gamma_t(T')=\gamma_t(T)-1$. So, $\gamma_{t,coi}(T')=\gamma_t(T')$, and by inductive hypothesis, $T' \in \mathcal{F}$. Also, it can relatively clearly be seen that $v$ is having distance three to a leaf. This means $T$ can be obtained from $T'$ by operation $F_3$, and so $T \in \mathcal{F}$.

Now, consider the case in which $u_1\in B$. By the maximality of $P(h,h')$ and by the fact that $|N(u_1)|\geq 3$, there is a leaf distinct to $h$ at distance two or three from $u_1$. Hence, the set $D'=D-\{s,v\}$ is a TC-ID set in $T'=T-\{h,s,v\}$. Again, by using a similar procedure as above (\textbf{Case 3.1} and $r\geq 2$) we obtain $T' \in \mathcal{F}$ and, due to that $T$ can be obtained from $T'$ by operation $F_4$, we get $T \in \mathcal{F}$.\\

\textbf{Case 3.3.3: $|N(v)\cap S(T)|=1$ and $|N(u_1)|=2$.} Clearly $s,v,u_1$ have degree two and $s,v$ belong to $D$.  We only consider the case whether $u_1\in B$, otherwise $u_1\in D$ implies that $u_2$ is a leaf and $T$ is $P_5$, which can be obtained by operation $F_2$ from $P_4$. As $u_1\in B$, we get $u_2\in D$. Notice that, as $u_2$ has to be totally dominated, there exist a vertex  $r\in D$ such that $N(r)\cap D=\{u_2\}$. So, by Remark \ref{theo-minimal} and Lemma \ref{dist-V-D}, it follows $(N(r)-\{u_2\})\subset L(T)\cup V_{2,3}(T)$.

If $(N(r)-\{u_2\})\cap L(T))\neq \emptyset$, then this case is analogous to the \textbf{Case 3.3.2} and $u_1\in B$. If $(N(r)-\{u_2\})\subset V_{2,3}(T)$, then we see that $u_1\in V_6(T)$. So, the set $D'=D-\{s,v\}$ is a TC-ID set in $T'=T-\{h,s,v\}$ and again, by using a similar procedure as above (\textbf{Case 3.1} and $r\geq 2$) we obtain $T' \in \mathcal{F}$. Finally, due to that $T$ can be obtained from $T'$ by operation $F_5$, we have $T \in \mathcal{F}$, which completes the proof.\\
\end{proof}

As an immediate consequence of Lemma \ref{lem-right} and Lemma \ref{lem-left} we have the following characterization.

\begin{theorem}\label{teo-lower}
Let $T$ be a tree. Then $T\in \mathcal{T}_{\gamma_t}$ if and only if $T \in \mathcal{F}$.
\end{theorem}

We next see that all the operations $F_1$ to $F_4$ are required in the characterization above. First, we see that operation $F_1$ is required to obtain a double star from the path $P_4$. The operations $F_2, F_3,F_4$ are required to obtain the paths $P_5,P_6,P_7$, respectively, from the path $P_4$, and the path $P_{10}$ can only be obtained from $P_4$ by a sequence of operations $F_4, F_5$.

\section*{Concluding remarks}

We have study several combinatorial and complexity properties of the total co-independent domination number of graphs. As a consequence of the study a couple of questions could be remarked as a possible future research lines.
\begin{itemize}
  \item We have proved that computing the total co-independent domination number of graphs is NP-hard even when restricted to planar graphs of maximum degree at most 3. However, it would be interesting to find some non trivial families of graphs in which the problem above can be solved in polynomial time. On the other hand, the bounds of Theorem \ref{th-cover-tcid} together with the fact that the problem of computing the vertex cover number can be approximated within a factor of 2, allow to claim that the problem of computing the total co-independent domination number can be approximated within a constant factor. In this sense, it would be interesting to give some other approximation (or inapproximation) results on this parameter.
  \item We have characterized the family of graphs achieving the upper bound of Theorem \ref{th-cover-tcid}. According to the construction of such family, it seems one could also characterize the graphs $G$ for which $\gamma_{t,coi}(G)=2\alpha(G)-k$ for some values of $k$ like for instance $k=2$ or $k=3$. Moreover, it would be of interest to characterize the family of graphs attaining the lower bound of Theorem \ref{th-cover-tcid} (note that for instance the trees satisfying such bound were characterized in \cite{Cabrera2017}).
\end{itemize}

\end{document}